\pdfoutput=1

\documentclass[11pt]{amsart}
\usepackage[paper,pkg/amsaddr=true,pkg/mathalfa=false,thm/numberwithin={}]{zbs}
\newcommand{\F}{\mathbb{F}}
\newcommand{\K}{\mathbb{K}}
\renewcommand{\a}{\alpha}
\renewcommand{\b}{\beta}
\DeclareMathOperator{\AR}{AR}
\DeclareMathOperator{\R}{R}
\DeclareMathOperator{\Q}{Q}
\newcommand{\g}{g}
\newcommand{\Nop}{N}
\newcommand{\Nff}[1]{\Nop_{#1}}

\newcommand{\Mult}[3]{\mathrm{Mult}^{#1}_{#2/#3}}

\title{Uniform stability of ranks}
\author[Guy Moshkovitz]{Guy Moshkovitz\nfts{1}}
\address{\nfts{1}Department of Mathematics, City University of New York (Baruch College \& Graduate Center), New York, NY 10010, USA}
\email{guymoshkov@gmail.com}
\author[Daniel G. Zhu]{Daniel G. Zhu\nfts{2}}
\address{\nfts{2}Department of Mathematics, Princeton University, Princeton, NJ 08544, USA}
\email{zhd@princeton.edu}

\thanks{The first author is supported by NSF Award DMS-2302988. The second author is supported by the NSF Graduate Research Fellowships Program (NSF grant DGE-2039656).}

\begin{document}
\begin{abstract}
    Chen and Ye recently proved that the analytic rank of tensors is stable under field extensions, assuming a fixed base field. Using a more careful analysis, we show that this assumption is unnecessary.
\end{abstract}

\maketitle

\section{Introduction}
Given a $d$-tensor $T \in \F^{n_1} \otimes_\F \cdots \otimes_\F \F^{n_d}$ over a finite field $\F$, let $\AR(T)$ denote its analytic rank~\cite{GW11}. Given a finite extension $\K/\F$, let $T^{\K} \in \K^{n_1} \otimes_\K \cdots \otimes_\K \K^{n_d}$ be the natural image of $T$. Chen and Ye~\cite{ChenYe} recently showed that the analytic rank of a tensor is stable under field extensions.
\begin{thm}[\cite{ChenYe}] \label{thm:ARstab}
For any $d \geq 2$ and prime power $q$, there exist constants $C(d,q)$ and $c(d,q)$ such that for any $d$-tensor $T$ over $\F_q$ and any finite field extension $\K/\F_q$, we have 
\[c(d,q) \AR(T) \leq \AR(T^\K) \leq C(d,q) \AR(T).\]
\end{thm}
The purpose of this note is to show that the constants in \cref{thm:ARstab} can be taken to be independent of $q$, a fact which has a few immediate consequences for the results of \cite{ChenYe}. For one, the analytic rank and the geometric rank are equivalent up to a constant independent of $q$. Moreover, Chen and Ye use \cref{thm:ARstab} to show the equivalence of several conjectures relating various notions of tensor rank; the independence of $q$ in \cref{thm:ARstab} in turn implies that if one of the aforementioned conjectures is true with constants independent of $q$, the same is true for the others.

For $d \geq 2$ and a finite extension $\K/\F$,\footnote{Recall that $\K$ forms a finite-dimensional vector space over $\F$, meaning $\K$ is isomorphic to $\F^n$ for some $n\ge 1$.} consider the
$\F$-multilinear map $\K^{d-1} \to \K$ given by $(d-1)$-ary multiplication,
and let $\Mult{d}{\K}{\F}$ denote the corresponding $d$-tensor over $\F$. Letting $\R(-)$ and $\Q(-)$ denote the rank\footnote{Also known as (classical) tensor rank, traditional rank, or cp-rank.} and subrank of tensors, respectively, Chen and Ye proved \cref{thm:ARstab} via the following lemma.
\begin{lem}[{\cite[extracted from proof of Theorem 4.2]{ChenYe}}]\label{prop:AR-stable}
    For any $d$-tensor $T$ over $\F_q$, 
    \[\frac{\Q(\Mult{d}{\F_{q^n}}{\F_q})}{n} \cdot \AR(T) \le \AR(T^{\F_{q^n}}) \le \frac{\R(\Mult{d}{\F_{q^n}}{\F_q})}{n} \cdot \AR(T).\]
\end{lem}
As a result, to accomplish our goal we will show the following.
\begin{thm} \label{thm:qr}
There exist constants $C_d = \exp(O(d\log\log d))$ and $c_d = \Omega(d^{-2})$ such that
\[\R(\Mult{d}{\F_{q^n}}{\F_q}) \leq C_d \cdot n \quad\text{ and }\quad \Q(\Mult{d}{\F_{q^n}}{\F_q}) \geq c_d \cdot n.\]
\end{thm}
We have not made a serious attempt to optimize these constants.

\section{Preliminaries}

Henceforth abbreviate $\R(\Mult{d}{\F_{q^n}}{\F_q})$ and $\Q(\Mult{d}{\F_{q^n}}{\F_q})$ as $R_d(n,q)$ and $Q_d(n,q)$, respectively. We begin with a few results relating these quantities. As the proof of \cref{lem:qmon} in \cite{ChenYe} is rather terse, we include a proof here.

\begin{lem}[{\cite[Lemma 1]{BalRol} and \cite[Lemma 3.6]{ChenYe}}] \label{lem:qrrel}
For a positive integers $m$ and $n$, we have
$R_d(n,q) \leq R_d(mn, q) \leq R_d(n, q^m) R_d(m, q)$ and $Q_d(mn, q) \geq Q_d(n, q^m) Q_d(m, q)$.
\end{lem}

\begin{lem}[{\cite[Lemma 3.5]{ChenYe}}] \label{lem:qmon}
If $n-1 \geq (d-1)(m-1)$, then $Q_d(n,q) \geq Q_d(m,q)$.
\end{lem}
\begin{proof}
Write $\F_{q^m} = \F_q[\alpha]$ and $\F_{q^n} = \F_q[\beta]$. Define linear maps $f \colon \F_{q^m} \to \F_{q^n}$ and $g \colon \F_{q^n} \to \F_{q^m}$ such that $f(\alpha^i) = \beta^i$ for $i \leq m-1$ and $g(\beta^i) = \alpha^i$ for $i \leq n-1$.

Put $k=d-1$. We claim that for every $x_1, \ldots, x_k \in \F_{q^m}$, we have
\[g(f(x_1) f(x_2) \cdots f(x_k)) = x_1 x_2 \cdots x_k.\]
Since both sides are $\F_q$-multilinear in $x_1,\ldots,x_{k}$,
equality holds if and only if it holds over the basis $\set{1,\a,\ldots,\a^{m-1}}$.
And indeed, for all choices $x_i = \a^{r_i}$ with $r_i \le m-1$, and $r = \sum_{i=1}^k r_i$, we have 
\[g(f(x_1)\cdots f(x_k)) = g(\b^r) = \a^r = x_1\cdots x_k ,\]
where the second equality uses $r \le k(m-1) \le n-1$ by the statement's assumption.
This implies that $\Mult{d}{\F_{q^m}}{\F_q}$ is a restriction 
of $\Mult{d}{\F_{q^n}}{\F_q}$, so the result follows.
\end{proof}

By a function field $K$ we mean an algebraic function field of one variable. We denote by $\g(K)$ its genus, and by $\Nff{n}(K)$ the number of places of $K$ of degree $n$.
For formal definitions, see \cite{Stichtenoth}. However, knowledge of these definitions is not necessary to follow the remainder of this paper, as all results using them will be suitably black-boxed.

\begin{lem}[\cite{BalRol}]\label{lemma:R-base-bound}
Suppose $K$ is a function field over $\F_q$ and $n>1$, $d>1$, and $N$ are positive integers such that
\begin{enumerate}[label=(\alph*)]
\item $\Nff{1}(K) \geq \g(K) + 1$
\item $(d-1)(n + \g(K) - 1) < N$
\item $\Nff{1}(K) \geq N$
\item $\Nff{n}(K) \geq 1$.
\end{enumerate}
Then $R_d(n, q) \leq N$.
\end{lem}
\begin{proof}
Setting $r=1$ in \cite[Theorem 2]{BalRol} and incorporating Remark 1 yields this result with $N = \Nff{1}(K)$. However, examining the proof reveals that the existence of extra places of degree $1$ does not affect the argument.
\end{proof}

\begin{lem}[\cite{ChenYe}]\label{lemma:Q-base-bound}
Suppose $K$ is a function field over $\F_q$ and $n>1$, $d>1$, and $N$ are positive integers such that
\begin{enumerate}[label=(\alph*)]
\item $\Nff{1}(K) \geq \g(K) + 1$
\item $(d-1)(N + \g(K) - 1) < n$
\item $\Nff{1}(K) \geq N$
\item $\Nff{n}(K) \geq 1$.
\end{enumerate}
Then $Q_d(n, q) \geq N$.
\end{lem}
\begin{proof}
This is Lemma 3.4 in \cite{ChenYe} except for the fact that condition (a) appears in the original as $N \geq \g(K)+1$. However, only the weaker condition that $\Nff{1}(K) \geq \g(K)+1$ is used.
\end{proof}

\begin{lem}[{\cite[Corollary 5.2.10]{Stichtenoth}}]\label{lemma:high-deg-place}
If $K$ is a function field over $\F_q$ and $n$ is a positive integer such that $\Nff{1}(K) \geq 1$ and $2\g(K)+1 \leq q^{n/2}-q^{(n-1)/2}$, then $\Nff{n}(K) \geq 1$.
\end{lem}

\begin{lem}[\cite{GS}]\label{lemma:tower}
If $\ell$ is a prime power and $i$ is a nonnegative integer, there exists a function field $K_{\ell,i}$ over $\F_{\ell^2}$ with $\g(K_{\ell,i}) < \ell^i$
and $\Nff{1}(K_{\ell,i}) \geq \ell^i(\ell-1)$.
\end{lem}
\begin{proof}
If $i = 0$, replace it with $i=1$; as we will show that in fact $\g(K_{\ell,1}) = 0$, this is fine.

If $i \geq 1$, \cite[Section 3]{GS} constructs a function field $T_i$ such that $[T_i : \F_{\ell^2}(x)] = \ell^{i-1}$, which we take to be $K_{\ell, i}$. By \cite[Remark 3.8]{GS}, we have $\g(K_{\ell,i}) = (\ell^{\ceil{i/2}} - 1)(\ell^{\floor{i/2}} - 1)$, which is both strictly less than $\ell^i$ and zero when $i = 1$. By \cite[Lemma 3.9]{GS}, there exist $\ell^2-\ell$ degree-one places of $\F_{\ell^2}(x)$ that totally split in $K_{\ell,i}$, so $\Nff{1}(K_{\ell,i}) \geq (\ell^2-\ell)[K_{\ell,i} : \F_{\ell^2}(x)] = \ell^i(\ell-1)$.
\end{proof}
Note that when applying \cref{lemma:R-base-bound,lemma:Q-base-bound} to these $K_{\ell,i}$, condition (a) is automatically satisfied. Also, \cref{lemma:high-deg-place} shows that $K_{\ell,i}$ has places of all degrees at least $i+2$, since
\[2\g(K_{\ell,i}) + 1 \leq 2\ell^i \leq \ell^i(\ell^2-\ell) = (\ell^2)^{(i+2)/2} - (\ell^2)^{(i+1)/2}.\]

\section{Proofs}
We will use the following basic fact, which is easy to show by casework.
\begin{fact}\label{fact:intervals}
    Suppose $a,b \in \setz$ and $x,y\in \setr$ such that $a \leq b$, $a \le y$, $x \le b$, and $y-x \ge 1$.
    Then $[a,b] \cap [x,y]$ contains an integer. 
\end{fact}

\begin{prop} \label{lem:r}
If $q \geq 64d^2$ is a square, then $R_d(n, q) \leq 8d^2n$.
\end{prop}
\begin{proof}
Let $q = \ell^2$ where $\ell \geq 8d$. The $n=1$ case is trivial, so assume $n\geq 2$.

Choose $i \geq 0$ such that $\ell^i < 4dn \leq \ell^{i+1}$.
By \cref{fact:intervals}, there exists a (positive) integer
\[N \in [2dn,8d^2n-1] \cap [2d\ell^i, \ell^{i+1}/2].\]
Indeed, $2dn \leq \ell^{i+1}/2$, $2d\ell^i < 8d^2n$, and $\ell^{i+1}/2-2d\ell^i = \ell^i(\ell/2 - 2d) \ge 2d\ell^i \ge 1$.

We now apply \cref{lemma:R-base-bound} with $N$ as above and $K = K_{\ell,i}$. Recall that condition (a) is automatically satisfied. Condition (b) is satisfied since
\[N = N/2 + N/2 \geq dn + d\ell^i > (d-1)(n+\g(K)-1).\]
Condition (c) is satisfied since $\Nff{1}(K) \geq \ell^i(\ell-1) \geq \ell^{i+1}/2 \geq N$. To check condition (d), we only need to show that $n \geq i+2$. If not, then
\[n > \frac{N}{8d^2} \geq \frac{2d\ell^i}{8d^2} = \frac{\ell^i}{4d} \geq \frac{\ell^{n-1}}{4d} \geq \frac{\ell}{4d} 2^{n-2} \geq 2^{n-1} \geq n,\]
a contradiction.
We deduce that $R_d(n, q) \le N < 8d^2 n$, as desired.
\end{proof}

\begin{prop} \label{lem:q}
If $q$ is a square, then $Q_d(n,q) \geq n/(4d)$.
\end{prop}
\begin{proof}
Let $q = \ell^2$. The $n < 4d$ case is trivial, so assume $n \ge 4d$.

Choose $i \geq 0$ such that $2d\ell^i \leq n < 2d\ell^{i+1}$. 
By \cref{fact:intervals}, there exists a (positive) integer
\[N \in [\ell^i, \ceil{\ell^{i+1}/2}] \cap [n/(4d), n/(2d)].\]
Indeed, $\ell^i \le n/(2d)$, $n/(4d) \leq \ceil{\ell^{i+1}/2}$, 
and $n/(2d) - n/(4d) = n/(4d) \ge 1$.

We now apply \cref{lemma:Q-base-bound} with $N$ as above and $K = K_{\ell, i}$. Recall that condition (a) is automatically satisfied. Condition (b) is satisfied since
\[(d-1)(N + \g(K) - 1) < d(N + \ell^i) \leq 2dN \leq n.\]
Condition (c) is satisfied as $\Nff{1}(K) \geq \ell^i(\ell-1) \geq \ceil{\ell^{i+1}/2} \geq N$. To check condition (d), we only need to show that $n \geq i+2$; this is clear as $n \geq 2d\ell^i \geq 2^{i+2} \geq i+2$.
We deduce that $Q_d(n, q) \ge N \ge n/(4d)$, as desired.
\end{proof}

One can lift \cref{lem:r} to general finite fields by using \cref{lem:qrrel} to bound $R_d(n,q) \leq R_d(2n, q) \le R_d(n, q^2) R_d(2, q)$. 
However, no analogous argument for subrank exists, as in some sense it would require that every field contains a square field, which is false.
To nevertheless lift \cref{lem:q} to general finite fields, we instead use the approximate monotonicity property stated in \cref{lem:qmon}.

\begin{proof}[Proof of \cref{thm:qr}]
To bound the rank, let $r$ be the smallest even integer such that $q^r \geq 64d^2$; 
note that $r \le 2\log_q d + 8$.
Then \cref{lem:qrrel,lem:r} yield
\[R_d(n,q) \leq R_d(rn,q) \leq R_d(n,q^r) \cdot R_d(r,q) \leq r^{d-1} \cdot 8d^2n \le C_{d} \cdot n\]
with $C_{d} = 8d^2(2 \log_2 d + 8)^{d-1}$.

To bound the subrank, we claim that $c_d = 1/(8d^2)$ 
works. If $n \leq d^2$, the statement is trivial. 
Otherwise, let $m=\ceil{n/2d}$, and apply \cref{lem:qmon} to obtain $Q_d(n,q) \geq Q_d(2m,q)$, as \[(d-1)(2m-1) \le (d-1)(n/d + 1) = n-1 -(n/d-d) \le n-1.\]
Together with \cref{lem:qrrel,lem:q}, we deduce 
\[Q_d(n,q) \geq Q_d(2m,q) \geq Q_d(m,q^2) \geq \frac{m}{4d} \geq \frac{n}{8d^2} = c_d \cdot n.\qedhere\]
\end{proof}

\printbibliography
\end{document}